\documentclass[a4paper, 10pt, american]{amsart}
\usepackage{geometry}
\usepackage{times,latexsym,amssymb}
\usepackage{amsmath,amsthm,bm}
\usepackage{color}
\usepackage[colorlinks,pdfpagelabels,pdfstartview = FitH,bookmarksopen
= true,bookmarksnumbered = true,linkcolor = blue,plainpages =
false,hypertexnames = false,citecolor = red,pagebackref=false]{hyperref}
\usepackage{soul}
\usepackage{microtype} 
\usepackage{amsbsy}
\usepackage{amstext}
\usepackage{amssymb}
\usepackage{esint}
\usepackage{stmaryrd}
\usepackage[pdftex]{graphicx}
\usepackage{floatflt}
\usepackage{comment}
\usepackage{cancel}

\allowdisplaybreaks
\newtheorem{mytheorem}{Theorem}
\newtheorem{mylem}[mytheorem]{Lemma}
\newtheorem{mydef}[mytheorem]{Definition}

\newtheorem{mycor}[mytheorem]{Corollary}
\newtheorem{remark}[mytheorem]{Remark}
\numberwithin{mytheorem}{section}
\numberwithin{equation}{section}

\SetSymbolFont{stmry}{bold}{U}{stmry}{m}{n}

\usepackage{schemata}
\usepackage{tikz}
\usetikzlibrary{decorations.pathreplacing,calc,tikzmark}
\usetikzlibrary{matrix,decorations.pathreplacing}

\usepackage{blindtext}

\newcommand{\uproman}[1]{\uppercase\expandafter{\romannumeral#1}}
\usepackage[utf8]{inputenc}
\usepackage{mathtools}

\DeclareMathOperator\divv{div}

\usepackage{mathrsfs}
\usepackage{yfonts}
\usepackage{braket}
\usepackage{bbm}

\usepackage{paralist}
\usepackage{pifont}
\usepackage{enumerate}

\usepackage{graphicx}
\newcommand{\bigchi}{\scalebox{1.3}{$\chi$}}

\makeatletter
\DeclareRobustCommand*{\bfseries}{%
  \not@math@alphabet\bfseries\mathbf
  \fontseries\bfdefault\selectfont
  \boldmath
}
\makeatother

\allowdisplaybreaks
\numberwithin{mytheorem}{section}
\numberwithin{equation}{section}

\hypersetup{linkcolor=blue,linktoc=page}
\usepackage{titletoc}

\def\YYint#1#2#3{{\setbox0=\hbox{$#1{#2#3}{\iint}$}
    \vcenter{\hbox{$#2#3$}}\kern-.51\wd0}}

\def\XXint#1#2#3{{\setbox0=\hbox{$#1{#2#3}{\int}$ }
\vcenter{\hbox{$#2#3$ }}\kern-0.555\wd0}}

\newcommand{\vertiii}[1]{{\left\vert\kern-0.25ex\left\vert\kern-0.25ex\left\vert #1 
    \right\vert\kern-0.25ex\right\vert\kern-0.25ex\right\vert}}

\renewcommand{\d}{\mathrm{d}}
\newcommand{\dx}{\mathrm{d}x}
\newcommand{\dy}{\mathrm{d}y}
\newcommand{\dz}{\mathrm{d}z}
\newcommand{\dt}{\mathrm{d}t}
\newcommand{\ds}{\mathrm{d}s}

\newcommand{\R}{\mathbb{R}}

\renewcommand{\epsilon}{\varepsilon}

\renewcommand{\d}{\mathrm{d}}

\renewcommand{\epsilon}{\varepsilon}


\begin{document}
\title[Comparison principle for doubly nonlinear parabolic fractional equations]{A comparison principle for a class of doubly nonlinear parabolic fractional partial differential equations}
\date{\today}


\author[M. Strunk]{Michael Strunk}
\address{Michael Strunk \\
Fachbereich Mathematik, Universit\"at Salzburg\\
Hellbrunner Str. 34, 5020 Salzburg, Austria}
\email{michael.strunk@plus.ac.at}


\begin{abstract} In this paper, we establish a comparison principle for non-negative weak solutions to a class of doubly nonlinear parabolic fractional partial differential equations within a space-time cylinder~$\Omega_T=\Omega\times(0,T)\subset\R^{n+1}$. For the two solutions considered, we assume that at least one of them is time-independent outside the spatial domain, i.e. in~$\Omega^{c}=\R^n\setminus\Omega$. As an application of this result, we readily infer the uniqueness of a non-negative weak solution to the corresponding Cauchy-Dirichlet problem. 
\end{abstract}


\maketitle
\vspace{-0.5cm}
\tableofcontents


\section{Introduction and main results}
We consider non-negative weak solutions to doubly nonlinear parabolic fractional partial differential equations of the type
\begin{align} \label{pde}
    \partial_t u^q - \mathcal{L}^s_{p,K}u = f \qquad\mbox{in~$\Omega_T$}, 
\end{align}
where~$s\in(0,1)$,~$p>1$, and~$q>0$. Here,~$\Omega_T\coloneqq\Omega\times(0,T)$ denotes a space-time cylinder taken over a bounded spatial domain~$\Omega\subset\R^n$ ($n\geq 2$) and some~$T>0$. For the inhomogeneity of equation~\eqref{pde}, we assume that it is time-independent and satisfies~$f\in L^{\sigma}(\Omega)$, where~$\widetilde{p}\coloneqq\min\{p,q+1\}$ and~$\sigma\coloneqq \frac{\widetilde{p}}{\widetilde{p}-1}$ denotes the Hölder conjugate of~$\widetilde{p}$. 
Moreover,~$\mathcal{L}^s_{p,K}$ denotes the nonlocal operator of equation~\eqref{pde} that is of fractional~$p$-Laplace-type and is given for a function~$u\colon\R^n\times(0,T)\to\R$ as
\begin{align} \label{kernel}
    \mathcal{L}^s_{p,K}u(x,t) \coloneqq 2\,\mbox{P.V.}\int_{\R^n} |u(x,t)-u(y,t)|^{p-2} (u(x,t)-u(y,t)) K(x,y)\, \dy
\end{align}
for~$(x,t)\in\R^n\times(0,T)$. Here, P.V. denotes the principle value of the integral. Moreover, the kernel~$K\colon \R^n\times\R^n\to[0,\infty)$ denotes a measurable function that satisfies the growth and symmetry condition
\begin{align} \label{kernelgrowth}
    \frac{\Lambda^{-1}}{|x-y|^{n+sp}} \leq K(x,y)=K(y,x) \leq \frac{\Lambda}{|x-y|^{n+sp}} 
\end{align}
for a.e.~$(x,y)\in \R^n\times\R^n$, with a positive structure constant~$\Lambda\geq 1$. 

The primary goal of this paper is to establish a comparison principle for non-negative weak solutions to equation \eqref{pde}, which is formalized in our main theorem, Theorem~\ref{hauptresultat}. Roughly speaking, the comparison principle in the local setting states that two solutions~$u,v$, defined in~$\Omega_T$, which satisfy~$u\leq v$ on the parabolic boundary~$\partial_p\Omega_T = \big(\overline{\Omega}\times\{0\} \big) \cup \big(\partial\Omega\times(0,T) \big)$, must have the same property in the entire space-time cylinder~$\Omega_T$, i.e. there holds~$u\leq v$ a.e. in~$\Omega_T$. Even though it is generally understood to be a rather simple property, the comparison principle for the local doubly nonlinear equation is still far from being fully understood; cf. Section~\ref{subsec:literature} for a more detailed analysis. In the nonlocal setting, the boundary condition must be adapted to the requirement that, for any two solutions~$u,v$ in~$\Omega_T$, the relation~$u\leq v$ holds a.e. in~$\Omega^c\times(0,T)$, where~$\Omega^c=\R^n\setminus\Omega$. Unsurprisingly, the nonlocal case - particularly the doubly nonlinear parabolic fractional equation \eqref{pde} - is even less well-understood than the local counterpart, and there are scarce results concerning a comparison principle. The main objective of this article is to address this gap in the existing literature.


\subsection{Literature overview} \label{subsec:literature}
Comparison principles in the local case for doubly nonlinear parabolic partial differential equations
\begin{align} \label{pdelocal}
    \partial_t u^q - \divv |Du|^{p-2}Du = 0 \qquad\mbox{in $\Omega_T$} 
\end{align} 
have been treated by a number of authors. It is important to note, however, that all existing comparison principles to date rely on additional assumptions that stem from the power-type nonlinearity~$q>0$ present in the evolutionary term of the equation. When~$q=1$, the result naturally follows from standard arguments without the need for further assumptions. Interestingly, the underlying proof does not depend on linear diffusion and remains essentially the same even in the nonlinear case where~$p\neq 2$. Bamberger \cite{Bamberger} proved a comparison principle for weak solutions to doubly nonlinear equations under the additional assumption that~$\partial_t u^q, \partial_t v^q \in L^1(\Omega_T)$. 
Alt and Luckhaus \cite{existence} obtained a comparison principle for weak sub- and super-solutions under a similar condition, assuming that there holds $(\partial_t u^q - \partial_t v^q) \in L^1(\Omega_T)$. Diaz \cite{Diaz} requires an additional assumption on the time derivative that is once again comparable to the aforementioned results. In particular, these assumptions are quite restrictive, since the definition of a weak solution does not include the existence of a weak time derivative and certain mollification methods have to be used in general. 

Otto~\cite{otto1996l1} employed an alternative proof technique to establish a comparison principle for weak sub- and super-solutions with time-independent boundary data on the lateral boundary. His method relies on a doubling of the time variable, a strategy also utilized in~\cite{bogelein2023h}. In this manuscript, we will adapt and extend this approach accordingly. 

Ivanov, Mkrtychan, and Jäger \cite{ivanov1997} followed yet another approach. They considered the case where $q\in(0,1]$ and $p\in (1,2)$ and proved a comparison principle for bounded and strictly positive sub- and super-solutions, i.e. they assumed the infimum of $u$ and $v$ to be strictly positive in the entire space-time domain $\Omega_T$. Their result was subsequently extended by Ivanov \cite{Ivanov_1995} to the range $q\in(0,1]$ and $p>1$. Lindgren and Lindqvist \cite{comparison} proved a comparison principle under similar additional assumptions for Trudinger's equation, i.e. the case where $q=p-1$. Recently, Anttila, Lindqvist, and Parviainen \cite{anttila2026uniqueness} also derived a comparison principle for sign-changing solutions
to Trudinger’s equation with continuous lateral boundary data under an additional technical assumption. The entire parameter space~$p>1$ and~$q>0$ was finally addressed by Bögelein and the author in \cite{bogelein2024comparison}, where a lower bound condition for the super-solution was imposed solely on the lateral boundary rather than throughout the entire space-time domain. 

In the nonlocal setting, even fewer results are available. Vázquez~\cite{vazquez2016dirichlet} established an~$L^1$-contraction property for solutions to the Cauchy-Dirichlet problem involving the fractional parabolic~$p$-Laplace equation, under homogeneous Dirichlet conditions in~$\Omega^c$ and with non-negative initial data. Similarly, Bonforte and Vázquez~\cite{bonforte2015priori} studied the Cauchy-Dirichlet problem for solutions to the fractional porous medium equation, also assuming homogeneous Dirichlet conditions in~$\Omega^c$. In their work, they proved a weighted~$L^1$-contraction property. To the best of our knowledge, there are no existing results concerning a comparison principle for the doubly nonlinear fractional parabolic equation~\eqref{pde}. Specifically, we aim to establish such a principle for more general Dirichlet boundary data, beyond the homogeneous case. 


\subsection{Definition of weak solution} \label{sec:weaksolution}
For $s\in(0,1)$ and a subset $\Omega \subset \R^n$, we introduce the fractional Sobolev space
$W^{s,p}(\Omega)$ defined by
$$ W^{s,p}(\Omega)
\coloneqq \bigg\{
v\in L^p(\Omega) \colon
\iint_{\Omega \times \Omega}
\frac{|v(x)-v(y)|^p}{|x-y|^{n+sp}}\dx\dy < \infty
\bigg\}, $$
which is equipped with the norm
$$ \|v\|_{W^{s,p}(\Omega)}
\coloneqq \bigg(\int_{\Omega} |v|^p\dx\bigg)^{\frac{1}{p}}
+ \bigg(\iint_{\Omega \times \Omega}
\frac{|v(x)-v(y)|^p}{|x-y|^{n+sp}}\dx\dy\bigg)^{\frac{1}{p}}. $$
Moreover, we denote
$$ W^{s,p}_0(\Omega) \coloneqq \big\{ v\in W^{s,p}(\R^n)\colon \mbox{$v=0$\,\, a.e. in~$\Omega^{c}=\R^n \setminus \Omega$} \big\} $$
as the space of~$W^{s,p}$-functions that vanish outside of~$\Omega$. At this point, we present our notion of a weak solution to equation~\eqref{pde}. 


\begin{mydef} \label{defweakform}
A measurable function $u\colon \R^n \times (0,T) \to \R_{\geq 0}$ satisfying
$$ u \in C (0, T; L^{q+1}(\Omega)) \cap L^p(0, T; W^{s,p}( \Omega) ) $$
is a non-negative weak solution to the equation \eqref{pde} in $\Omega_T$, if there holds
\begin{align} \label{tailglobal}
    \int_{0}^{T} \int_{\R^n} \frac{|u(x,t)|^{p-1}}{1+|x|^{n+sp}} \dx \dt < \infty 
\end{align}
and 
\begin{align} \label{weakform}
 \int_{0}^{T}\iint_{\R^n\times\R^n} & |u(x,t) - u(y,t)|^{p-2}(u(x,t)-u(y,t))(\phi(x,t)-\phi(y,t)) K(x,y)\,\dx\dy\dt \\
  &- \iint_{\Omega_T} u^q \, \partial_t\phi\, \dx \dt  = \iint_{\Omega_T} f\phi\,\dx\dt \nonumber
\end{align} 
for any test function
$$ \phi \in  L^p (0, T; W^{s,p}_0(\Omega)) \cap  W^{1,q+1}_0 (0, T; L^{q+1}(\Omega)). $$
\end{mydef}


\begin{remark} \upshape
   We observe that the nonlocal integral term in~\eqref{weakform} is indeed finite for any non-negative weak solution to~\eqref{pde}. This finiteness is ensured by the global integrability assumption~\eqref{tailglobal}; cf.~\cite{liao2024modulus} for details where the very same assumption was imposed. Importantly, the condition~\eqref{tailglobal} is solely required to establish the well-posedness of the weak formulation of~\eqref{pde} and it is not invoked elsewhere in this manuscript. 
\end{remark}


In what follows, we will consider the Cauchy-Dirichlet problem 
\begin{align} \label{cauchydirichlet}
    \begin{cases}
        \partial_t u^q - \mathcal{L}^s_{p,K} = f & \mbox{in~$\Omega_T$}, \\
        u=g & \mbox{in~$\Omega^{c}$}\times(0,T), \\
        u(\cdot,0) = u_0 & \mbox{in~$\Omega$} 
    \end{cases}
\end{align}
for non-negative measurable data~$g\colon\R^n\times(0,T)\to\R$ with~$g\in L^p\big(0,T;W^{s,p}(\Omega)\big)$ that satisfies the global integrability condition~\eqref{tailglobal}, and non-negative initial data~$u_0\in L^{q+1}(\Omega,\R_{\geq 0})$. Throughout this paper, the boundary values are taken in the sense that~$(u-g)\in L^p\big(0,T;W^{s,p}_0(\Omega)\big)$, whereas the initial data~$u_0$ is taken in the a.e.-sense. We observe that the boundary data~$g$ must be specified on the entire complement~$\Omega^c$ of the spatial domain, which differs from the local case where boundary conditions are only imposed on the actual boundary~$\partial\Omega$ of the spatial domain. 


\begin{mydef} \label{defcauchydirichlet}
    A measurable function $u\colon \R^n \times (0,T) \to \R_{\geq 0}$ satisfying
    $$ u \in C (0, T; L^{q+1}(\Omega)) \cap g+L^p(0, T; W^{s,p}_{0}( \Omega) ) $$
    is a non-negative weak solution to the Cauchy-Dirichlet problem~\eqref{cauchydirichlet}, if~$u$ is a non-negative weak solution to~\eqref{pde} in the sense of Definition~\ref{defweakform} and there holds~$u(\cdot,0)=u_0$ a.e. in~$\Omega$. 
\end{mydef}



\subsection{Main result}

In this section, we present the main result of our paper, which states a global comparison principle for non-negative weak solutions to~\eqref{pde} under the assumption that one of them is time-independent in~$\Omega^{c}=\R^n\setminus\Omega$. 


\begin{mytheorem} \label{hauptresultat} 
Let~$p>1$,~$q>0$,~$f\in L^{\sigma}(\Omega)$, where~$\widetilde{p}\coloneqq\min\{p,q+1\}$ and~$\sigma\coloneqq \frac{\widetilde{p}}{\widetilde{p}-1}$, and~$u,v$ denote two non-negative weak solutions to~\eqref{pde}. Moreover, assume that either~$u(x,t_1)=u(x,t_2)$ or~$v(x,t_1)=v(x,t_2)$ a.e. in~$\Omega^c$ for a.e.~$t_1,t_2\in(0,T)$. If
$$ (u-v)_+ \in L^p\big(0,T;W^{s,p}_0(\Omega) \big) $$
and
$$  u(\cdot,0) \leq v(\cdot,0) \qquad\mbox{a.e. in~$\Omega$}, $$
then  there holds
$$ u \leq v \qquad\mbox{a.e. in~$\Omega_T$}. $$
\end{mytheorem}






As an immediate consequence of Theorem~\ref{hauptresultat}, we obtain the following uniqueness result for the Cauchy-Dirichlet problem~\eqref{cauchydirichlet}. 


\begin{mycor} \label{corollaryzwei}
  Let~$p>1$,~$q>0$, and consider the data
  \begin{align*}
      \begin{cases}
          f\in L^{\sigma}(\Omega), \\
          g\in L^p\big(0,T; W^{s,p}(\Omega,\R_{\geq 0})\big), \\
          u_0\in L^{q+1}(\Omega,\R_{\geq 0}),
      \end{cases}
  \end{align*}
  where the measurable boundary data~$g\colon\R^n\times(0,T)\to\R_{\geq 0}$ satisfies the global integrability condition~\eqref{tailglobal} and also~$g(x,t_1)=g(x,t_2)$ a.e. in~$\Omega^c$ for a.e.~$t_1,t_2\in(0,T)$. Then, there exists a unique non-negative weak solution to the Cauchy-Dirichlet problem~\eqref{cauchydirichlet} under the given data. 
\end{mycor}

\begin{proof} 
    Let~$u_1,u_2$ denote two non-negative weak solutions to the Cauchy-Dirichlet problem~\eqref{cauchydirichlet}. Then, for the boundary data there holds
    $$ u_1-u_2 = (u_1-g)-(u_2-g)\in L^p\big(0,T;W^{s,p}_0(\Omega)\big), $$ 
    and for the initial data we have
    $$ u_1(\cdot,0)=u_2(\cdot,0)=u_0 \qquad\mbox{a.e. in~$\Omega$}. $$
    Since the boundary datum~$g$ is time-independent in~$\Omega^c$ for  a.e.~$t\in(0,T)$, according to the assumptions of the corollary, and there holds
    $$ u_1(\cdot,t)=u_2(\cdot,t)=g(\cdot,t) \qquad\mbox{a.e. in~$\Omega^c\times(0,T)$}, $$
    we may apply Theorem~\ref{hauptresultat}, which yields~$u_1 \leq u_2$ a.e. in~$\Omega_T$. By interchanging the roles of~$u_1$ and~$u_2$, we similarly obtain the reverse inequality~$u_1 \geq u_2$ a.e. in~$\Omega_T$, which yields~$u_1=u_2$ a.e. in~$\Omega_T$ and in turn implies the uniqueness of a non-negative weak solution to~\eqref{cauchydirichlet}. 
\end{proof}


\subsection{Strategy of the proof} \label{strategy}
Due to the power-nonlinearity~$q>0$ present in the evolutionary term of our equation~\eqref{pde}, a delicate analysis is required in order to circumvent the lack of a time derivative and still obtain a comparison principle. In fact, an appropriate choice of test function is not straightforward at all. We follow the approach taken in~\cite[Proposition~4.16]{bogelein2023h}, which is inspired by Otto's idea from~\cite{otto1996l1}, whereas the latter itself relies on the method of Kru{\v{z}}kov~\cite{kruvzkov1970first}. The technique of proof relies on a doubling of the time variable argument. Indeed, the time-independent nature of~$f$ and~$K$ are a crucial condition for the proof. Theorem~\ref{hauptresultat} builds upon two useful preliminary results, Lemma~\ref{hilfslemeins} and Lemma~\ref{hilfslemzwei}, that essentially take care of the power-nonlinearity~$q>0$ in an appropriate way. Lemma~\ref{hilfslemeins} provides an integral relation between the non-negative weak solution~$u$ and a stationary non-negative measurable function~$\Tilde{v}$. In a similar manner, Lemma~\ref{hilfslemzwei} states a comparable result for~$v$ and a stationary non-negative measurable function~$\Tilde{u}$. In the proof of the main result, Theorem~\ref{hauptresultat}, the two non-negative weak solutions~$u,v$ will take the role of~$\Tilde{u},\Tilde{v}$ in Lemma~\ref{hilfslemeins} resp. Lemma~\ref{hilfslemzwei}. At this point, it is crucial to note that this approach necessitates~$u,v$ to be weak solutions, rather than non-negative weak sub- and super-solutions, which usually is the common assumption in terms of comparison principles. Moreover, the proof of Theorem~\ref{hauptresultat} heavily relies on the assumption that either~$u$ or~$v$ exhibits a time-independent behavior a.e. in~$\Omega^c$, as stated within the theorem.  


\subsection{Plan of the paper} \label{planpaper}
The paper is organized as follows: in Section~\ref{sec:preliminaries}, we begin by presenting the notation and framework, including supplementary material required later on. In Section~\ref{sec:comparison}, we give the proof of the comparison principle stated in Theorem~\ref{hauptresultat}, our main result of the paper. 


\subsection*{Acknowledgements}
This research was funded in whole or in part by the Austrian Science Fund (FWF) [10.55776/P36295]. For open access purposes, the author has applied a CC BY public copyright license to any author accepted manuscript version arising from this submission. \,\\ 

\textbf{Conflict of Interest.} The author declares that there is no conflict of interest. \,\\ 

\textbf{Data availability.} This manuscript has no associated data. 


\section{Preliminaries} \label{sec:preliminaries}
\subsection{Notation and setting}

The following Lipschitz function on~$\R$, given by
\begin{align} \label{hdelta}
    \mathcal{H}_\delta(x) \coloneqq 
    \displaystyle\begin{cases}
    1 & \mbox{for $x\geq \delta$},\\
    \frac{x}{\delta} & \mbox{for $0< x < \delta$},\\
    0 & \mbox{for $x\leq 0$}, 
    \end{cases}
\end{align}
where~$\delta>0$, will be useful for the derivation of the comparison principle. Note that~$\mathcal{H}_\delta$ approximates the Heaviside function on~$\R$ when~$\delta\downarrow 0$. Moreover, we introduce
\begin{align} \label{inthdelta}
    \mathfrak{h}_\delta(r,s) \coloneqq \int_{s}^r \mathcal{H}_\delta(z-s) q z^{q-1}\,\dz, \qquad\mbox{$r,s\in\R_{\geq 0}$} 
\end{align}
as well as
\begin{align} \label{inthdeltareflection}
    \widehat{\mathfrak{h}}_\delta(r,s) \coloneqq \int_{s}^r \widehat{\mathcal{H}}_\delta(z-s) q z^{q-1}\,\dz, \qquad\mbox{$r,s\in\R_{\geq 0}$}, 
\end{align}
where~$\widehat{\mathcal{H}}_\delta(s) \coloneqq -\mathcal{H}_\delta(-s)$ for~$s\in\R$ denotes the odd reflection of~$\mathcal{H}_\delta$. The positive part of a real quantity~$a\in\R$ is denoted as~$a_+=\max\{a,0\}$. Finally, by~$\bigchi_A(x)$, we refer to the characteristic function for some set~$A\subset\R^n$.   
   

\subsection{Mollification in time} \label{subsec:mollificationintime} 
According to our definition, weak solutions may not necessarily exhibit weak differentiability with respect to the time variable~$t\in(0,T)$. In order to overcome this lack of regularity, we introduce a regularization technique involving the exponential function. Given~$v \in L^1(\Omega_T)$ and~$0<h<T$, we define mollified functions 
\begin{equation} \label{mollification}
	[v]_{h}(x,t) 
	\coloneqq
    \displaystyle{\mbox{$\frac{1}{h}$} \int_{0}^{t} e^{\frac{\tau-t}{h}} v(x,\tau) \,\d\tau }, \qquad [v]_{\Bar{h}}(x,t) 
	\coloneqq
    \displaystyle{\mbox{$\frac{1}{h}$} \int_{t}^{T} e^{\frac{t-\tau}{h}} v(x,\tau) \,\d\tau }. 
\end{equation} 

 We recall the most relevant properties here, where we refer to~\cite[Appendix B]{bogelein2013parabolic} for additional information.


\begin{mylem} \label{lem:timemol} 
For any  $r\geq 1$ we have:
\begin{itemize}
    \item[(\upshape i)] If $u\in L^r(\Omega_T)$, then $[u]_h \in L^r(\Omega_T)$. Moreover, 
    $\| [u]_h\|_{L^r(\Omega_T)}\leq \|u\|_{L^r(\Omega_T)}$ and $ [ u ]_h \to u$
    in $L^r(\Omega_T)$ and a.e. in $\Omega_T$ as $h\downarrow 0$. 
    \item[(\upshape ii)] $[u]_{h}\in C\big(0,T; L^r(\Omega)\big)$.
    \item[(\upshape iii)] There holds
    \begin{equation*}
	 \partial_t [ u ]_h
	 =
	 \tfrac{1}{h} \big(u-[ u ]_h\big), \quad
	  \partial_t [ u ]_{\bar{h}}
	 =
	 \tfrac{1}{h} \big([ u ]_{\bar{h}}- u\big)
\end{equation*}
a.e. in $\Omega_T$. 
    \item[(\upshape iv)] If $Du\in L^r(\Omega_T)$, then $D[ u ]_{h} = [ Du ]_{h}\to Du$
    in $L^r(\Omega_T)$ and a.e. in $\Omega_T$ as $h\downarrow 0$. 
    \item[(\upshape v)] If $u\in C\big(0,T; L^r(\Omega)\big) $, then $[u]_{h}(\cdot ,t)\to u(\cdot,t)$
    in $L^r(E)$ and a.e. in $\Omega$  for every $t\in [0,T]$ as $h\downarrow 0$. 
\end{itemize}
For $[ u ]_{\bar{h}}$, there hold similar conclusions as in  {\rm (i)}, {\rm(ii)}, {\rm(iv)}, and {\rm (v)}.
\end{mylem}



\section{Comparison principle} \label{sec:comparison}

In this section, we aim to establish our main result, that is Theorem~\ref{hauptresultat}. As the first step in this direction, we begin with the following lemma. 


\begin{mylem} \label{hilfslemeins} 
    Let~$p>1$,~$q>0$, and~$u$ be given as in Theorem~\ref{hauptresultat}. Moreover, assume that~$\widetilde{v}\colon\R^n\to\R_{\geq 0}$ is measurable with~$\widetilde{v} \in W^{s,p}(\Omega)$. 
    If
    $$ u(\cdot,t)\leq \widetilde{v} \qquad\mbox{a.e. in~$\Omega^c$} $$
    for a.e.~$t\in(0,T)$, then there holds
    \begin{align} \label{est:hilfslemeinsest}
        \int_{0}^{T}\iint_{\R^n\times\R^n} & |u(x,t) - u(y,t)|^{p-2}(u(x,t)-u(y,t))(\psi(x,t) - \psi(y,t) ) K(x,y)\,\dx\dy\dt \\
         &- \iint_{\Omega_T} \mathfrak{h}_\delta(u,\widetilde{v})\partial_t\phi\,\dx\dt = \iint_{\Omega_T} f\mathcal{H}_\delta(u - \widetilde{v})\phi\,\dx\dt  \nonumber, 
    \end{align}
    for any~$\phi\in C^\infty_0(0, T)$. Here,
    $$ \psi(\cdot,t)\equiv \mathcal{H}_\delta(u(\cdot,t)-\widetilde{v}(\cdot))\phi(t) \qquad\mbox{a.e. in~$\R^n$} $$ 
    for a.e.~$t\in(0,T)$.
\end{mylem}

\begin{proof}
We test the weak form~\eqref{weakform} satisfied by~$u$ with~$\zeta = \mathcal{H}_\delta([u]_h - \widetilde{v})\phi$, where~$\phi\in C^\infty_0(0, T)$ is chosen arbitrarily. Indeed, this choice of test function is admissible since there holds
$$ [u]_h\leq u\leq v \qquad\mbox{a.e. in~$\Omega^c$} $$
 for a.e.~$t\in(0,T)$ and~$h>0$ small enough, such the Dirichlet condition~$([u]_h-v)_+\in L^p\big(0,T;W^{s,p}_0(\Omega)\big)$ is satisfied. Due to the definition of~$\mathcal{H}_\delta$, we further have that $\mathcal{H}_\delta(u(\cdot,t)-\widetilde{v}(\cdot)) \in W^{s,p}_0(\Omega)$ for a.e.~$t\in(0,T)$. For the evolutionary term of the equation, we use the property (iii) of Lemma~\eqref{lem:timemol} and the fact that~$\widetilde{v}$ is independent of the time variable, and estimate
\begin{align*}
    \iint_{\Omega_T} - u^q \partial_t \zeta\,\dx\dt &= \iint_{\Omega_T} ( -[u]^q_h+[u]^q_h - u^q) \partial_t \zeta\,\dx\dt \\
    &= \iint_{\Omega_T} - [u]^q_h \partial_t \zeta\,\dx\dt + \iint_{\Omega_T} ([u]^q_h-u^q) \mathcal{H}_\delta([u]_h - \widetilde{v}) \partial_t \phi\,\dx\dt \\
    &\quad - \iint_{\Omega_T} ([u]^q_h - u^q) \mathcal{H}'_\delta([u]_h-\widetilde{v}) h^{-1}([u]_h-u) \phi\,\dx\dt \\
    &\leq \iint_{\Omega_T} \partial_t [u]^q_h \zeta \,\dx\dt + \iint_{\Omega_T}([u]^q_h-u^q) \mathcal{H}_\delta([u]_h - \widetilde{v}) \partial_t \phi\,\dx\dt \\
    &= - \iint_{\Omega_T} \mathfrak{h}_\delta([u]_h,\widetilde{v})\partial_t \phi\,\dx\dt + \iint_{\Omega_T}([u]^q_h-u^q) \mathcal{H}_\delta([u]_h - \widetilde{v}) \partial_t \phi\,\dx\dt. 
\end{align*}
Since~$u\in C\big(0,T;L^{q+1}(\Omega)\big)$, we have in particular that~$u\in L^{q+1}(\Omega_T)$, such that the last integral term vanishes in the limit~$h\downarrow 0$ due to an application of Hölder's inequality with exponents~$(q+1,\frac{q+1}{q})$; cf. Lemma~\ref{lem:timemol} (i). Moreover, the definition of~$\mathfrak{h}_\delta$ yields
\begin{align*}
    |\mathfrak{h}_\delta(u,\widetilde{v})-\mathfrak{h}_\delta([u]_h,\widetilde{v})| = \bigg| \int_{[u]_h}^{u} \mathcal{H}_\delta(s-\widetilde{v})qs^{q-1}\,\ds \bigg| \leq |u^q-[u]^q_h|,
\end{align*}
such that another application of the convergence property in Lemma~\ref{lem:timemol} (i) implies 
\begin{align*}
    -\iint_{\Omega_T} \mathfrak{h}_\delta([u]_h,\widetilde{v})\partial_t\phi\,\dx\dt \to - \iint_{\Omega_T} \mathfrak{h}_\delta(u,\widetilde{v})\partial_t\phi\,\dx\dt
\end{align*}
as~$h\downarrow 0$. For the nonlocal term in the weak form of~$u$, there holds
\begin{align*}
    \int_{0}^{T} & \iint_{\R^n\times\R^n}  |u(x,t) - u(y,t)|^{p-2}(u(x,t)-u(y,t))(\zeta(x,t)-\zeta(y,t)) K(x,y)\,\dx\dy\dt \\
    &\to \int_{0}^{T}\iint_{\R^n\times\R^n} |u(x,t) - u(y,t)|^{p-2}(u(x,t)-u(y,t))(\psi(x,t) - \psi(y,t) ) K(x,y)\,\dx\dy\dt
\end{align*}
as~$h\downarrow 0$, which can be inferred in a similar manner as for example in~\cite[Proof of Lemma~4.3]{misawa2025expansion}. Finally, for the right-hand side of the equation we have
$$ \iint_{\Omega_T} f\mathcal{H}_\delta([u]_h - \widetilde{v})\phi\,\dx\dt \to \iint_{\Omega_T} f\mathcal{H}_\delta(u - \widetilde{v})\phi\,\dx\dt $$
as~$h\downarrow 0$. Combining our results, we infer an inequality in the desired identity~\eqref{est:hilfslemeinsest} in form of~"$\geq$". In order to also obtain the reverse inequality, we test the weak form of~$u$ with~$\zeta=\mathcal{H}_\delta([u]_{\Bar{h}}-\widetilde{v})\phi$ and proceed in a similar manner as before. 
\end{proof}


Via a similar reasoning to Lemma~\ref{hilfslemeins}, we readily obtain a second lemma. 


\begin{mylem} \label{hilfslemzwei} 
     Let~$p>1$,~$q>0$, and~$v$ be given as in Theorem~\ref{hauptresultat}. Moreover, assume that~$\widetilde{u}\colon\R^n\to\R_{\geq 0}$ with~$\widetilde{u} \in L^{q+1}(\Omega) \cap W^{s,p}(\Omega)$. 
       If
    $$ \widetilde{u}\leq v(\cdot,t) \qquad\mbox{a.e. in~$\Omega^c$} $$
    for a.e.~$t\in(0,T)$, then, there holds
    \begin{align} \label{est:hilfslemzweiest}
        \int_{0}^{T}\iint_{\R^n\times\R^n} & |v(x,t) - v(y,t)|^{p-2}(v(x,t)-v(y,t))(\widehat{\psi}(x,t) - \widehat{\psi}(y,t) ) K(x,y)\,\dx\dy\dt \\
         &- \iint_{\Omega_T} \widehat{\mathfrak{h}}_\delta(v,\widetilde{u})\partial_t\phi\,\dx\dt = \iint_{\Omega_T} f \widehat{\mathcal{H}}_\delta(v - \widetilde{u})\phi\,\dx\dt \nonumber, 
    \end{align}
    for any~$ \phi\in C^\infty_0(0, T)$. Here,
    $$ \widehat{\psi}(\cdot,t)\equiv \widehat{\mathcal{H}}_\delta(v(\cdot,t)-\widetilde{u}(\cdot))\phi(t) \qquad\mbox{a.e. in~$\R^n$} $$
    for a.e.~$t\in(0,T)$. 
\end{mylem}

\begin{proof}
 In the weak form~\eqref{weakform} satisfied by~$v$, we use the test function~$\zeta=\widehat{\mathcal{H}}_{\delta}([v]_h-[\widetilde{u}]_h)\phi$, where~$ \phi\in C^\infty_0(0, T)$. A few words on this choice of test function are in order here. Even though~$\widetilde{u}$ is independent of the time variable~$t\in(0,T)$, we still choose the test function in such a way that it involves the quantity~$[\widetilde{u}]_h$ rather than~$\widetilde{u}$. The reason for this lies in the fact that the mollification~$[\cdot]_h$ not only has a smoothing effect, but also a deflating one that shrinks the function. For this matter, the assumption that~$\widetilde{u}\leq v(\cdot,t)$ does not suffice to ensure that there holds~$\widetilde{u}\leq [v]_h$ a.e. in~$\Omega^c$ for a.e.~$t\in(0,T)$, 
 which excludes the more natural choice of test function~$\zeta=\widehat{\mathcal{H}}_{\delta}([v]_h-\widetilde{u})\phi$. Moreover, we have that if~$\widetilde{u}\leq v$ a.e. in~$\Omega^c$ for a.e.~$t\in(0,T)$, then there also holds~$[\widetilde{u}]_h\leq [v]_h$ a.e. in~$\Omega^c$ for a.e.~$t\in(0,T)$. Now, according to the definition of~$\widehat{\mathcal{H}}$, our choice of test function is indeed admissible. 
 
 The nonlocal term of the equation can be dealt with in the same manner as in Lemma~\ref{hilfslemeins}. For the evolutionary term of equation~\eqref{pde}, we use the fact that~$\Tilde{u}$ is independent of the time variable, which yields
 $$ \Tilde{u}-[\Tilde{u}]_h = e^{-\frac{t}{h}}\Tilde{u} \qquad\mbox{a.e. in~$\Omega_T$} $$
 for~$h>0$. Exploiting this property, we estimate
 \begin{align*}
    \iint_{\Omega_T} - v^q \partial_t \zeta\,\dx\dt &= \iint_{\Omega_T} ( -[v]^q_h+[v]^q_h - v^q) \partial_t \zeta\,\dx\dt \\
    &= \iint_{\Omega_T} - [v]^q_h \partial_t \zeta\,\dx\dt + \iint_{\Omega_T} ([v]^q_h-v^q) \widehat{\mathcal{H}}_\delta([v]_h-[\widetilde{u}]_h ) \partial_t \phi\,\dx\dt \\
    &\quad + \iint_{\Omega_T} ([v]^q_h - v^q) \mathcal{H}'_\delta([\widetilde{u}]_h-[v]_h) h^{-1}(\widetilde{u}-[\widetilde{u}]_h-(v-[v]_h)) \phi\,\dx\dt \\
    &\geq \iint_{\Omega_T} \partial_t [v]^q_h \zeta \,\dx\dt + \iint_{\Omega_T}([v]^q_h-v^q) \widehat{\mathcal{H}}_\delta([v]_h - [\widetilde{u}]_h) \partial_t \phi\,\dx\dt \\
    &\quad + \iint_{\Omega_T} ([v]^q_h - v^q) \mathcal{H}'_\delta([\widetilde{u}]_h-[v]_h) h^{-1} e^{-\frac{t}{h}} \widetilde{u} \phi\,\dx\dt \\
    &= - \iint_{\Omega_T} \widehat{\mathfrak{h}}_\delta([v]_h,[\widetilde{u}]_h)\partial_t \phi\,\dx\dt - \iint_{\Omega_T} \int_{[\widetilde{u}]_h}^{[v]_h} \partial_t \widehat{\mathcal{H}}_\delta(s-[\widetilde{u}]_h)q s^{q-1}\,\ds \dx\dt  \\
    &\quad + \iint_{\Omega_T}([v]^q_h-v^q) \widehat{\mathcal{H}}_\delta([v]_h - [\widetilde{u}]_h) \partial_t \phi\,\dx\dt \\
    &\quad + \iint_{\Omega_T} ([v]^q_h - v^q) \mathcal{H}'_\delta([\widetilde{u}]_h-[v]_h) h^{-1} e^{-\frac{t}{h}} \widetilde{u} \phi\,\dx\dt \\
    &= - \iint_{\Omega_T} \widehat{\mathfrak{h}}_\delta([v]_h,[\widetilde{u}]_h)\partial_t \phi\,\dx\dt \\
    &\quad  + \iint_{\Omega_T}\int_{[v]_h}^{[\widetilde{u}]_h}  \delta^{-1}h^{-1}e^{-\frac{t}{h}}\widetilde{u}(x,t) \bigchi_{ \{[\widetilde{u}]_h-\delta<s<[\widetilde{u}]_h\} }  q s^{q-1}\,\ds\dx\dt  \\
    &\quad + \iint_{\Omega_T}([v]^q_h-v^q) \widehat{\mathcal{H}}_\delta([v]_h - [\widetilde{u}]_h) \partial_t \phi\,\dx\dt \\
    &\quad + \iint_{\Omega_T} ([v]^q_h - v^q) \mathcal{H}'_\delta([\widetilde{u}]_h-[v]_h) h^{-1} e^{-\frac{t}{h}} \widetilde{u} \phi\,\dx\dt. 
\end{align*}
In turn, we used the Leibniz integral rule after the first inequality sign in the preceding calculation. We recall that~$v\in C\big(0,T;L^{q+1}(\Omega) \big)$, which implies in particular that~$v\in L^{q+1}(\Omega_T)$. For the first term on the right-hand side, we have the convergence
$$ \iint_{\Omega_T} \widehat{\mathfrak{h}}_\delta([v]_h,[\widetilde{u}]_h)\partial_t \phi\,\dx\dt \to \iint_{\Omega_T} \widehat{\mathfrak{h}}_\delta(v,\widetilde{u})\partial_t \phi\,\dx\dt $$
as~$h\downarrow 0$, which can be inferred in a similar manner to the previous lemma, Lemma~\ref{hilfslemeins}. The second term on the right-hand side of the preceding inequality vanishes in the limit~$h\downarrow 0$, which follows from the dominated convergence theorem. Similarly, the third term vanishes when~$h\downarrow 0$, due to Lemma~\ref{lem:timemol} (i). Moreover, also the last term on the right-hand side vanishes, since~$\mathcal{H}'_\delta$ and also~$\phi$ are bounded, and we may employ Hölder's inequality with exponents~$(q+1,\frac{q+1}{q})$ and the dominated convergence theorem to pass to the limit~$h\downarrow 0$. Finally, the right-hand side of the equation for~$v$ is dealt with as in Lemma~\ref{hilfslemeins}. Thus, we have verified the inequality in the desired identity~\eqref{est:hilfslemzweiest} with~"$\leq$". In order to obtain the reverse inequality, we test the weak form of~$v$ with~$\zeta=\widehat{\mathcal{H}}_\delta([v]_{\Bar{h}}-[\widetilde{u}]_{\Bar{h}})\phi$ and proceed similarly. 
\end{proof}


We are in position to provide the proof of our main result, Theorem~\ref{hauptresultat}, stating the comparison principle for non-negative weak solutions to~\eqref{pde}. 


\begin{proof}[\textbf{\upshape Proof of Theorem~\ref{hauptresultat}}]

Let~$\delta>0$. For the doubling of the time variable, we abbreviate
$$ Q \coloneqq \Omega\times(0,T)^2 $$
and choose a non-negative function~$\phi\in C^{\infty}_0((0,T)^2)$. We extend~$u$ and~$v$ to~$Q$ by setting
$$ u(x,t_1,t_2) \coloneqq u(x,t_1), \qquad v(x,t_1,t_2)\coloneqq v(x,t_2) $$
for a.e.~$t_1,t_2\in(0,T)$. Next, we apply Lemma~\ref{hilfslemeins} with the choice~$\widetilde{v}(x)=v(x,t_2)\eqqcolon v_{t_2}(x)$ and also~$(0,T)\ni t_1 \mapsto \phi(t_1,t_2)\eqqcolon \phi_{t_2}(t_1)$ for a.e.~$t_2\in(0,T)$. Similarly, we apply its counterpart Lemma~\ref{hilfslemzwei} with the choice~$\widetilde{u}(x)=u(x,t_1)\eqqcolon u_{t_1}(x)$ and~$(0,T)\ni t_2 \mapsto \phi(t_1,t_2)\eqqcolon \phi_{t_1}(t_2)$ for a.e.~$t_1\in(0,T)$. In turn, we used the fact that, according to Definition~\ref{defweakform}, there holds~$u,v\in L^p\big(0,T;W^{s,p}(\Omega)\big)$, such that~$v_{t_2}$ resp.~$u_{t_1}$ belong to the correct elliptic space for a.e.~$t_1,t_2\in(0,T)$. Moreover, since we assumed that~$(u-v)_+\in L^p\big(0,T;W^{s,p}_0(\Omega)\big)$, and also that one of them is time-independent a.e. in~$\Omega^c$ for a.e.~$t\in(0,T)$, we in particular can ensure that~$u_{t_1}(x)\leq v_{t_2}(x)$ for a.e.~$x\in\Omega^c$ and a.e.~$t_1,t_2\in(0,T)$, such that Lemma~\ref{hilfslemeins} resp. Lemma~\ref{hilfslemzwei} are indeed applicable. This way, we obtain 
\begin{align} \label{est:hilfslemeinsapplied}
    \int_{0}^{T}\iint_{\R^n\times\R^n} & |u(x,t_1) - u(y,t_1)|^{p-2}(u(x,t_1)-u(y,t_1))(\psi(x,t_1) - \psi(y,t_1)) K(x,y)\,\dx\dy\dt_1 \\
    &- \iint_{\Omega_T} \mathfrak{h}_\delta(u,v_{t_2})\partial_{t_1}\phi_{t_2}\,\dx\dt_1 = \iint_{\Omega_T} f \mathcal{H}_\delta(u-v_{t_2})\phi \,\dx\dt_1 \nonumber 
\end{align}
and
\begin{align} \label{est:hilfslemzweiapplied}
    \int_{0}^{T}\iint_{\R^n\times\R^n} & |v(x,t_2) - v(y,t_2)|^{p-2}(v(x,t_2)-v(y,t_2))(\widehat{\psi}(x,t_2) - \widehat{\psi}(y,t_2)) K(x,y)\,\dx\dy\dt_2 \\
    &- \iint_{\Omega_T} \widehat{\mathfrak{h}}_\delta(v,u_{t_1})\partial_{t_2}\phi_{t_1}\,\dx\dt_2 = \iint_{\Omega_T} f \widehat{\mathcal{H}}_\delta(v-u_{t_1})\phi \,\dx\dt_2. \nonumber 
\end{align}
Here, we set~$\psi(\cdot,t_1)\equiv \mathcal{H}_\delta(u(\cdot,t_1)-v_{t_2}(\cdot))\phi_{t_2}(t_1)$ and~$\widehat{\psi}(\cdot,t_2)\equiv \widehat{\mathcal{H}}_\delta(v(\cdot,t_2)-u_{t_1}(\cdot))\phi_{t_1}(t_2)$ a.e. in~$\R^n$ for a.e.~$t_1,t_2\in(0,T)$. Moreover,~$u$ is to be understood as a function of~$x,t_1$ and~$v$ as a function of~$x,t_2$. Next, we integrate the identity~\eqref{est:hilfslemeinsapplied} with respect to~$t_2\in(0,T)$, integrate~\eqref{est:hilfslemzweiapplied} with respect to~$t_1\in(0,T)$, use the fact that~$\widehat{\mathcal{H}}_\delta(s)=-\mathcal{H}_\delta(-s)$ for any~$s\in\R$ and also that the kernel~$K(x,y)$ is independent of the time variable and symmetric with respect to~$x,y\in\R^n$, according to~\eqref{kernelgrowth}, and finally add both identities. This way, the sum of the two right-hand sides of~\eqref{est:hilfslemeinsapplied} and~\eqref{est:hilfslemzweiapplied} vanishes due to~$f$ also being time-independent. In order to simplify notation, we omit the dependence on the respective time variable in the summed quantity. Thus, we infer
\begin{align} \label{est:summed}
    &\iiiint_{(0,T)^2\times(\R^n)^2}  \big[ |u(x)-u(y)|^{p-2}(u(x)-u(y)) - |v(x)-v(y)|^{p-2}(v(x)-v(y)) \big] K(x,y) \\
    &\qquad\qquad\qquad\quad\cdot\big[ \mathcal{H}_\delta(u(x)-v(x))-\mathcal{H}_\delta(u(y)-v(y)) \big]\phi \,\dx\dy\dt_1\dt_2 \nonumber \\
    &\quad- \iiint_{Q} \big( \mathfrak{h}_\delta(u,v)\partial_{t_1}\phi+\widehat{\mathfrak{h}}_\delta(v,u)\partial_{t_2}\phi \big) \,\dx\dt_1\dt_2 = 0. \nonumber 
\end{align}
Passing to the limit~$\delta\downarrow 0$ in the nonlocal integral term and using the fact that~$\mathcal{H}_\delta(s) \to \bigchi_{\{s\geq 0\}}$ pointwise on~$\R$ as~$\delta\downarrow 0$, as well as the symmetry of $K$, we obtain
\begin{align} \label{est:nonlocalterm}
    \lim\limits_{\delta\downarrow 0} &\iiiint_{(0,T)^2\times(\R^n)^2}  \big[ |u(x)-u(y)|^{p-2}(u(x)-u(y)) - |v(x)-v(y)|^{p-2}(v(x)-v(y)) \big] K(x,y) \\
    &\qquad\qquad\qquad\quad\cdot\big[ \mathcal{H}_\delta(u(x)-v(x))-\mathcal{H}_\delta(u(y)-v(y)) \big]\phi \,\dx\dy\dt_1\dt_2 \nonumber \\
    &= \iiiint_{(0,T)^2\times(\R^n)^2}  \big[ |u(x)-u(y)|^{p-2}(u(x)-u(y)) - |v(x)-v(y)|^{p-2}(v(x)-v(y)) \big] K(x,y) \nonumber \\
    &\qquad\qquad\qquad\qquad\cdot\big[ \bigchi_{\{u(x)\geq v(x)\}} -\bigchi_{\{u(y)\geq v(y)\}} \big]\phi \,\dx\dy\dt_1\dt_2  \nonumber \\
    &= \iiiint_{(0,T)^2\times\{u(x)\geq v(x)\}\times\{u(y)\leq v(y)\}} |u(x)-u(y)|^{p-2}(u(x)-u(y))K(x,y)\phi\,\dx\dy\dt_1\dt_2  \nonumber \\
    & \quad - \iiiint_{(0,T)^2\times\{u(x)\geq v(x)\}\times\{u(y)\leq v(y)\}} |v(x)-v(y)|^{p-2}(v(x)-v(y)) K(x,y)\phi \,\dx\dy\dt_1\dt_2  \nonumber \\
    & \quad - \iiiint_{(0,T)^2\times\{u(x)\leq v(x)\}\times\{u(y)\geq v(y)\}} |u(x)-u(y)|^{p-2}(u(x)-u(y)) K(x,y)\phi \,\dx\dy\dt_1\dt_2  \nonumber \\
    & \quad + \iiiint_{(0,T)^2\times\{u(x)\leq v(x)\}\times\{u(y)\geq v(y)\}} |v(x)-v(y)|^{p-2}(v(x)-v(y)) K(x,y)\phi \,\dx\dy\dt_1\dt_2  \nonumber \\
    & = 2 \iiiint_{(0,T)^2\times\{u(x)\geq v(x)\}\times\{u(y)\leq v(y)\}}  |u(x)-u(y)|^{p-2}(u(x)-u(y)) K(x,y)\phi \,\dx\dy\dt_1\dt_2  \nonumber \\
    & \quad -  2 \iiiint_{(0,T)^2\times\{u(x)\geq v(x)\}\times\{u(y)\leq v(y)\}} |v(x)-v(y)|^{p-2}(v(x)-v(y)) K(x,y)\phi \,\dx\dy\dt_1\dt_2.  \nonumber
\end{align}
We claim that this quantity is non-negative, where we follow an idea from~\cite[Theorem~5.2]{bahrouni2018comparison}. For this, we use the fundamental theorem of calculus for integrals (if~$p\geq 2$) resp. for improper integrals (if~$1<p<2$) as follows
\begin{align*}
    F(b)-F(a) = (b-a) \int_{0}^{1}F'(a+s(b-a))\,\ds, \qquad a,b\in\R,
\end{align*}
where~$F(s) \coloneqq |s|^{p-2}s$ for~$s>0$. Due to~$p>1$, we let $F(0)\coloneqq 0$ with~$F'(s)=|s|^{p-2}$, and apply the fundamental theorem with the choice 
$$ a=v(x)-v(y), \quad b=u(x)-u(y). $$
This allows us to write
\begin{align*}
    |&u(x) - u(y)|^{p-2}(u(x)-u(y)) - |v(x)-v(y)|^{p-2}(v(x)-v(y)) \\
    &= (p-1) \int_{0}^1 |(v(x)-v(y)) + s(u(x)-u(y)-(v(x)-v(y)))|^{p-2}\,\ds\, (u(x)-u(y)-(v(x)-v(y))) \\
    &\eqqcolon L(x,y) (u(x)-u(y)-(v(x)-v(y)))
\end{align*}
for a.e.~$x,y\in\{u(x)\geq v(x)\}\times\{u(y)\leq v(y)\}$. We note that~$L(x,y)$ and, moreover, also
$$ L(x,y) (u(x)-u(y)-(v(x)-v(y))) $$
is non-negative for a.e.~$x,y\in\{u(x)\geq v(x)\}\times\{u(y)\leq v(y)\}$. In turn, due to the non-negativity of the kernel~$K(x,y)$ and the test function~$\phi$, the quantity within~\eqref{est:nonlocalterm} is indeed non-negative, such that it can be discarded. \\ 

Next, let us deal with the evolutionary term in~\eqref{est:summed}. First, we notice that
$$ 0\leq \mathfrak{h}_\delta(r,s) \leq \int_{s}^{r} qz^{q-1}\,\dz = r^q-s^q $$
and also 
\begin{align*}
    \lim\limits_{\delta\downarrow 0} \mathfrak{h}_\delta(r,s) &= \lim\limits_{\delta\downarrow 0} \int_{s}^{r}\mathcal{H}_\delta(z-s)qz^{q-1}\,\dz \\
    &= \int_{s}^{r}\bigchi_{\{z\geq s\}}qz^{q-1}\,\dz \\
    &= r^q-s^q
\end{align*}
for any non-negative~$s\leq r$. Moreover, the definition of~$\mathcal{H}_\delta$ implies that~$\mathfrak{h}_\delta(r,s)=0$ for any~$s>r$, such that we establish
$$ 0\leq \mathfrak{h}_\delta(r,s) \leq (r^q-s^q)_+,\qquad \lim\limits_{\delta\downarrow 0} \mathfrak{h}_\delta(r,s) = (r^q-s^q)_+ $$
for any non-negative~$r,s\geq 0$. In a similar manner, one verifies that there holds
$$ 0\leq \widehat{\mathfrak{h}}_\delta(r,s) \leq (s^q-r^q)_+,\qquad \lim\limits_{\delta\downarrow 0} \widehat{\mathfrak{h}}_\delta(r,s) = (s^q-r^q)_+. $$
Passing to the limit in~\eqref{est:summed} and taking the preceding reasoning concerning the nonlocal integral term into account thus yields
\begin{align} \label{est:afterlimit}
    - \iiint_{Q} (u^q-v^q)_+(\partial_{t_1}\phi+\partial_{t_2}\phi) \,\dx\dt_1\dt_2 \leq 0
\end{align}
for any non-negative~$\phi\in W^{1,q+1}_0((0,T)^2)$. At this point, we let~$\epsilon>0$ and choose~$\phi$ as
$$ \phi(t_1,t_2)=\phi_\epsilon(t_1,t_2)\coloneqq \epsilon^{-1}\psi\Big(\frac{t_1-t_2}{\epsilon} \Big) \zeta\Big(\frac{t_1+t_2}{2}\Big), $$
where~$\psi\in C^\infty_0(\R)$ denotes a non-negative smooth function with the property that~$\int_{\R}\psi \,\dt = 1$ and~$\zeta\in C^\infty_0(0,T)$ is non-negative. This test function is admissible, provided that~$\epsilon>0$ is chosen small enough. A calculation verifies
\begin{align*}
    \partial_{t_1}\phi_\epsilon+\partial_{t_2}\phi_\epsilon = \epsilon^{-1}\psi\Big(\frac{t_1-t_2}{\epsilon} \Big) \zeta'\Big(\frac{t_1+t_2}{2}\Big)
\end{align*}
for any~$t_1,t_2\in(0,T)$. Plugging~$\phi_\epsilon$ into~\eqref{est:afterlimit}, performing a change of variables via setting~$s=t_1-t_2$, and renaming~$t_1$ to~$t$, there holds
\begin{align} \label{est:varchange}
    - \int_{\R} \epsilon^{-1}\psi\Big(\frac{s}{\epsilon} \Big) \iint_{\Omega_T} (u^q(x,t)-v^q(x,t-s))_+ \zeta'\Big(t-\frac{s}{2}\Big)\,\dx\dt\ds \leq 0. 
\end{align}
For the inner integral term, we have
\begin{align*}
    \lim\limits_{s\to 0} \iint_{\Omega_T} & (u^q(x,t)-v^q(x,t-s))_+ \zeta'\Big(t-\frac{s}{2}\Big)\,\dx\dt \\
    & = \iint_{\Omega_T} (u^q-v^q)_+ \zeta'\,\dx\dt, 
\end{align*}
which implies that after passing to the limit~$s\to 0$ and also~$\epsilon\downarrow 0$ in~\eqref{est:varchange} yields
\begin{align*}
    -\iint_{\Omega_T} (u^q-v^q)_+\zeta'\,\dx\dt \leq 0
\end{align*}
for any non-negative function~$\zeta\in C^\infty_0(0,T)$. 

As the last step of the argument, we remove the assumption that~$\zeta$ vanishes at initial time~$t=0$. Through a density argument, we may choose~$\eta$ as a Lipschitz function as follows: we write~$\zeta = \omega\eta$, where~$\eta\in W^{1,\infty}_0((-\infty,T);[0,1])$ and~$\omega$ is chosen in such a way that it vanishes on~$(-\infty,0]$, is equal to one on~$[\delta,\infty)$ for~$\delta>0$ and interpolates linearly on~$(0,\delta)$. Plugging this choice of~$\zeta$ into the preceding integral inequality, we obtain 
\begin{align*}
   - \iint_{\Omega_T} (u^q-v^q)_+\eta'\omega\,\dx\dt &\leq \iint_{\Omega_T} (u^q-v^q)_+\eta\omega'\,\dx\dt \\
   &= \frac{1}{\delta} \iint_{\Omega\times(0,\delta)} (u^q-v^q)_+ \eta \,\dx\dt \\
   &\leq \frac{1}{\delta} \iint_{\Omega\times(0,\delta)} (u^q-v^q)_+ \,\dx\dt \\
   & \to \int_{\Omega\times\{0\}} (u^q - v^q)_+ \,\dx \\
   &=0
\end{align*}
as~$\delta\downarrow 0$. The last equality holds due to our initial condition~$u(\cdot,0)\leq v(\cdot,0)$ a.e. in~$\Omega$. On the left-hand side we pass to the limit~$\delta\downarrow 0$ to achieve 
\begin{align*}
    - \iint_{\Omega_T} (u^q-v^q)_+\eta'\,\dx\dt \leq 0. 
\end{align*}
Finally,~$\eta$ is chosen in such a way that it satisfies~$\eta(0)=1$,~$\eta(T)=0$, and~$\eta'(t)=-\frac{1}{T}$ for~$t\in(0,T)$. This leads to
\begin{align*}
     \iint_{\Omega_T} (u^q-v^q)_+ \,\dx\dt \leq 0,
\end{align*}
which implies the desired conclusion
\begin{align*}
    u \leq v \qquad\mbox{a.e. in~$\Omega_T$}
\end{align*}
and finishes the proof.

\end{proof}



 \nocite{*}
 \,\\
\bibliographystyle{plain}
\bibliography{Literatur.bib}

@article{bogelein2013parabolic,
  title={Parabolic systems with $p, q$-growth: a variational approach},
  author={B{\"o}gelein, V. and Duzaar, F. and Marcellini, P.},
  journal={Archive for Rational Mechanics and Analysis},
  volume={210},
  number={1},
  pages={219--267},
  year={2013},
  publisher={Springer}
}

@article{bogelein2023h,
  title={Hölder Continuity of the Gradient of Solutions to Doubly Non-Linear Parabolic Equations},
  author={B{\"o}gelein, V. and Duzaar, F. and Gianazza, U. and Liao, N. and Scheven, C.},
  journal={arXiv preprint arXiv:2305.08539},
  year={2023}
}

@article{misawa2025expansion,
  title={Expansion of positivity for the doubly nonlinear parabolic fractional type equations},
  author={Misawa, M. and Yamaura, Y.},
  journal={Mathematische Annalen},
  volume={393},
  number={2},
  pages={2561--2629},
  year={2025},
  publisher={Springer}
}

@article{otto1996l1,
  title={${L}^1$-contraction and uniqueness for quasilinear elliptic--parabolic equations},
  author={Otto, F.},
  journal={Journal of Differential Equations},
  volume={131},
  number={1},
  pages={20--38},
  year={1996},
  publisher={Elsevier}
}

@article{bahrouni2018comparison,
  title={Comparison and sub-supersolution principles for the fractional $p (x)$-{L}aplacian},
  author={Bahrouni, A.},
  journal={Journal of Mathematical Analysis and Applications},
  volume={458},
  number={2},
  pages={1363--1372},
  year={2018},
  publisher={Elsevier}
}

@article{bogelein2024comparison,
  title={A comparison principle for doubly nonlinear parabolic partial differential equations},
  author={B{\"o}gelein, V. and Strunk, M.},
  journal={Annali di Matematica Pura ed Applicata (1923-)},
  volume={203},
  number={2},
  pages={779--804},
  year={2024},
  publisher={Springer}
}

@article{Bamberger,
  title={{\'E}tude d'une {\'e}quation doublement non lin{\'e}aire},
  author={Bamberger, A.},
  journal={Journal of functional analysis},
  volume={24},
  number={2},
  pages={148--155},
  year={1977},
  publisher={Elsevier}
}

@article{Ivanov_1995,
  title={Existence and uniqueness of a regular solution of the {C}auchy-{D}irichlet problem for doubly nonlinear parabolic equations},
  author={Ivanov, A. V.},
  journal={Zeitschrift f{\"u}r Analysis und ihre Anwendungen},
  volume={14},
  number={4},
  pages={751--777},
  year={1995}
}

@article{Diaz,
  title={Qualitative {S}tudy of {N}onlinear {P}arabolic {E}quations: an {I}ntroduction.},
  author={Diaz, J. I.},
  journal={Extracta Mathematicae},
  volume={16},
  number={3},
  pages={303--341},
  year={2001}
}

@article{ivanov1997,
  title={Existence and uniqueness of a regular solution of the {C}auchy-{D}irichlet problem for a class of doubly nonlinear parabolic equations},
  author={Ivanov, A. V. and Mkrtychan, P. and J{\"a}ger, W.},
  journal={Journal of Mathematical Sciences},
  volume={1},
  number={84},
  pages={845--855},
  year={1997}
}

@article{comparison,
  title={On a comparison principle for {T}rudinger’s equation},
  author={Lindgren, E. and Lindqvist, P.},
  journal={Advances in Calculus of Variations},
  year={2020},
  publisher={De Gruyter}
}

@article{existence,
  title={Quasilinear {E}lliptic-{P}arabolic {D}ifferential {E}quations},
  author={Alt, H. W. and Luckhaus, S.},
  journal={Math. Z.},
  volume={183},
  number={3},
  pages={311--341},
  year={1983}
}

@article{anttila2026uniqueness,
  title={Uniqueness of sign-changing solutions to {T}rudinger's equation},
  author={Anttila, R. and Lindqvist, P. and Parviainen, M.},
  journal={arXiv preprint arXiv:2602.04748},
  year={2026}
}

@article{vazquez2016dirichlet,
  title={The {D}irichlet problem for the fractional $p$-{L}aplacian evolution equation},
  author={V{\'a}zquez, J. L.},
  journal={Journal of Differential Equations},
  volume={260},
  number={7},
  pages={6038--6056},
  year={2016},
  publisher={Elsevier}
}

@article{kruvzkov1970first,
  title={First order quasilinear equations in several independent variables},
  author={Kru{\v{z}}kov, S. N.},
  journal={Mathematics of the USSR-Sbornik},
  volume={10},
  number={2},
  pages={217--243},
  year={1970}
}

@article{bonforte2015priori,
  title={A priori estimates for fractional nonlinear degenerate diffusion equations on bounded domains},
  author={Bonforte, M. and V{\'a}zquez, J. L.},
  journal={Archive for Rational Mechanics and Analysis},
  volume={218},
  number={1},
  pages={317--362},
  year={2015},
  publisher={Springer}
}

@article{liao2024modulus,
  title={On the modulus of continuity of solutions to nonlocal parabolic equations},
  author={Liao, N.},
  journal={Journal of the London Mathematical Society},
  volume={110},
  number={3},
  pages={e12985},
  year={2024},
  publisher={Wiley Online Library}
}


\end{document}